\documentclass[12pt,twoside]{amsart}
\usepackage{amsmath, amsthm, amscd, amsfonts, amssymb, graphicx, color}
\usepackage[bookmarksnumbered, plainpages]{hyperref}

\usepackage[utf8]{inputenc}
\usepackage{pgf,tikz}
\usetikzlibrary{arrows}
\newtheorem{thm}{Theorem}[section]
\newtheorem{cor}[thm]{Corollary}
\newtheorem{lem}[thm]{Lemma}

\newtheorem{defn}[thm]{Definition}
\newtheorem{rem}[thm]{\bf{Remark}}

\numberwithin{equation}{section}
\usepackage[utf8]{inputenc}
\usepackage{pstricks-add}

\date{}

\title{\bf General Form of the Automorphism Group of Bicyclic Graphs }

\author{\bf Somayeh Madani and Ali Reza Ashrafi$^{\star}$ }

\address{Department of Pure Mathematics, Faculty of Mathematical Science,
University of Kashan, Kashan 87317-53153, I. R. Iran}

\thanks{$^{\star}$Corresponding author (Email: ashrafi@kashanu.ac.ir)}

\begin{document}

\maketitle

\begin{abstract}
In 1869, Jordan proved that the set $\mathcal{T}$ of all finite group that can be represented as the automorphism group of a tree is containing the trivial group and it is closed under taken direct product of groups of lower order in $\mathcal{T}$ and wreath product of a member in $\mathcal{T}$ and the symmetric group on $n$ symbols. The aim of this paper is to continue this work and another works by Klav$\acute{\rm i}$k and Zeman in 2017 to present a class $\mathcal{S}$ of finite groups for which the automorphism group of each bicyclic graph is a member of $\mathcal{S}$ and this class is minimal with this property.

\vskip 3mm

\noindent\textbf{Keywords:} Automorphism group, tree, unicyclic graph, bicyclic graph.

\vskip 3mm

\noindent\textit{2010 AMS Subject Classification Number:} 20B25.
\end{abstract}

\section{Basic Definitions}

The aim of this section is to provide some introductory materials that will
be kept throughout. All graphs are assumed to be undirected, simple and finite. The set of all vertices and edges of a graph $G$ are denoted by $V(G)$ and $E(G)$, respectively. A rooted graph is a graph in which one vertex is distinguished as the root. If the graph $\Gamma$ is containing $n$ vertices and $m$ edges, then the cyclomatic number of $G$ is defined as $c = m - n + 1$. If $c = 0, 1, 2$ then $G$ is called a tree, a unicyclic and a bicyclic graph, respectively.

Suppose $G$ and $H$ are groups and $K$ acts on a set $X$. Define: $$\{ (h;f) \mid f:X \longrightarrow G \ \& \ h \in H\} \ \ ; \ \ (h_1;f_1)(h_2;f_2) = (h_1h_2;f_1^{h_2}f_2),$$ where $f_1^{h_2}(x) = f_1(x^{h_2})$. This product defines a group structure and the resulting group is called the wreath product of $G$ with $H$, denoted by $G \wr H$. The wreath product is an important tool to describe the automorphism group of graphs. Let the connected components of a graph $G$ consist of $n_1$ copies of $G_1$, $n_2$ copies of $G_2$, $\ldots$, $n_r$ copies of $G_r$, where $G_1$, $\ldots$, $G_r$ are pairwise non-isomorphic. Then by a well-known result of Jordan $Aut(G) \cong (Aut(G_1) \wr S_{n_1}) \times \cdots \times (Aut(G_r) \wr S_{n_r}).$

Suppose $G_1$, $G_2$  and $G_3$ are  graphs with disjoint vertex sets and $v_1 \in V(G_1)$, $w_1$ $\in$ $V(G_2)$, $v_2 \in V(G_1 \cup G_2) \setminus \{ v_1, w_1 \}$ and $w_2 \in V(G_3)$. The union $G_1 \cup G_2$ is a graph with vertex set $V(G_1) \cup V(G_2)$ and edge set $E(G_1) \cup E(G_2)$. The graph union of more than two graphs can be defined inductively.  Following Do$\check{\rm s}$li\'c \cite{1x}, the splice $S(G_1,G_2;v_1,w_1)$ is defined by identifying the vertices $v_1$ and $w_1$ in $G_1 \cup G_2$. In a similar way, $S(G_1,G_2,G_3;v_1,w_1;v_2,w_2)$ $=$ $S(S(G_1,G_2;v_1,w_1),G_3;v_2,w_2)$ and we can define the splice of more than two graphs with respect to a parent graph by an inductive method. The link $L(G_1,G_2;v,w)$ is defined by adding an edge to the union graph $G_1 \cup G_2$ connecting the vertices $v$ and $w$. The link of more than two graphs can be defined similar to the splice.

Suppose $G$ is a simple and undirected graph and $u, v \in V(G)$. The distance between $u$ and $v$ is defined as the  length of a shortest path connecting these vertices. The eccentricity $\varepsilon(v)$  is defined to be the greatest distance between $v$ and any other vertices of $G$. The center of $G$ is the set of all vertices with minimum eccentricity, i.e the set of all vertices $u$ such that the greatest distance $d(u,v)$ to other vertices $v$ is minimal.

All calculations of this paper are done with the aid of GAP \cite{8} and Mathematica \cite{7}. We refer to  \cite{2,6} for basic definitions and notations not presented here.

\section{Backgrounds}

Suppose $\mathcal{C}$ is a class of graphs and $Aut(\mathcal{C})$ denotes the set of all groups that can be presented as the automorphism group of a member in $\mathcal{C}$. If $C$ is the class of all trees then $\mathcal{C}$ is denoted by
$\mathcal{T}$. By a result of Jordan \cite{3},   $\mathcal{T}$ is the class of all finite groups that can be defined inductively as follows:
\begin{enumerate}
\item $\{ 1\} \in \mathcal{T}$;
\item if $G_1, G_2 \in \mathcal{T}$, then $G_1 \times G_2 \in \mathcal{T}$;
\item if $G \in \mathcal{T}$ and $n \geq 2$, then $G \wr S_n \in \mathcal{T}$.
\end{enumerate}

One of the most interesting results after Jordan is a result of Babai \cite{1}. To state this result, we assume that $X$ and $Y$ are graphs and $f: V(X) \longrightarrow V(Y)$ is a mapping between vertex sets of $X$ and $Y$. The function $f$ is called a contraction if ($i$) $y_1y_2 \in E(Y)$ if and only if $y_1 \ne y_2$ and there is an edge $x_1x_2 \in E(X)$ such that $f(x_1) = y_1$ and $f(x_2) = y_2$; ($ii$) for any $y \in V(Y)$, the induced subgraph of $X$ on $f^{-1}(y)$ is connected. In the mentioned paper, Babai proved that if $\mathcal{C}$ is a class of finite graphs with this property that  $\mathcal{C}$ is closed under contraction and forming subgraphs, and if every finite  group occurs as the
automorphism group of a graph in  $\mathcal{C}$, then  $\mathcal{C}$ contains all finite graphs up to isomorphism. In another paper \cite{15}, he proved that  if $G$ is planar, then the group $Aut(G)$ has a sub normal chain $Aut(G) \triangleleft Y_1 \triangleleft Y_2 \triangleleft \cdots  \triangleleft Y_m = 1$.

Set $\mathcal{I} = \{ Aut(U) \mid U \ is \ an \ interval \ graph \}$. Klav$\acute{\rm i}$k and Zeman \cite{5}  proved that  $\mathcal{T}$ $=$  $\mathcal{I}$. They also obtained some interesting relation between the set of automorphism groups of some known classes of graphs. We encourage the interested readers to consult \cite{55} for more information on this problem.

The aim of this paper is to continue the interesting works of Klav$\acute{\rm i}$k and Zeman  by computing the automorphism group of bicyclic graphs.  In an exact phrase, if $\mathcal{S}$ denotes the set of all groups in the form of $Aut(G)$ with bicyclic graph $G$ then the set $\mathcal{S}$ will be determined in general.

\section{Main Results}
The aim of this section is to compute the automorphism group of an arbitrary bicyclic graph. To do this, we define:
\begin{eqnarray*}
\mathcal{B}_1 &=& \{ C \times (D \wr (Z_{2} \times Z_2)) \mid C, D \in \mathcal{T}\},\\
\mathcal{B}_2 &=& \{ C \times [(D \times D \times D \times D \times H \times H \times K \times K) \rtimes  (Z_{2} \times Z_2)] \mid C, D, H, K \in \mathcal{T}\},
\end{eqnarray*}
and $\mathcal{S} = \mathcal{T} \cup \mathcal{B}_1 \cup \mathcal{B}_2$. In this section, it will be proved that $\mathcal{S}$ is the set of all groups in the form of $Aut(\Delta)$, when $\Delta$ is a bicyclic graph.

Suppose $T$ is a tree, $G$ is a group, $X$ is a set and $u \in V(T)$. A branch at $u$ in $T$ is a maximal subtree containing $u$ as an endpoint, see \cite[p. 35]{35}. If the group $G$ acts on  $X$ and $x \in X$ then $G_x$ denotes the stabilizer subgroup of $G$ at the point $x$. An asymmetric graph is one with trivial automorphism group.

\begin{figure}
\centering
\includegraphics[width=10cm]{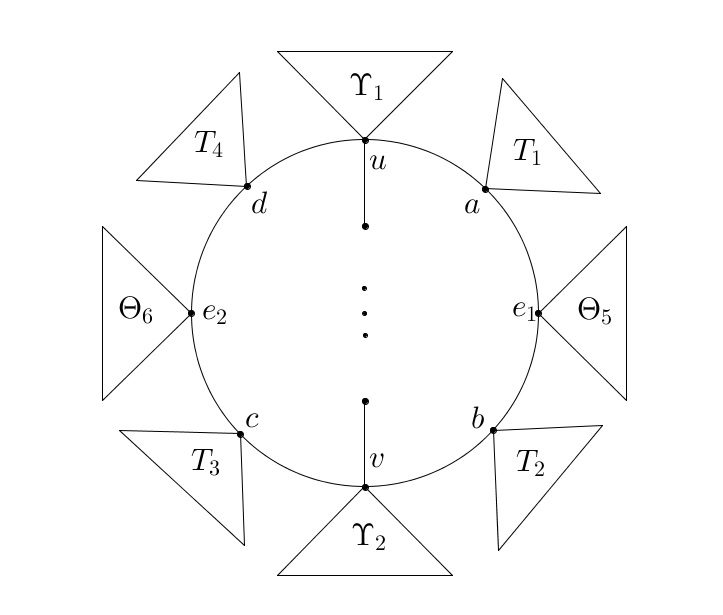}
\caption{The bicyclic graph of Lemma \ref{lem2}.} \label{1}
\end{figure}

The following simple lemma will be useful in our calculations.

\begin{lem}
Suppose $T$ is a tree, $G = Aut(T)$ and $v \in V(T)$. Then $G_v \in \mathcal{T}$.
\end{lem}

\begin{proof}
Choose an asymmetric tree $\Lambda$ containing a pendent vertex $w$ such that the degree of the  unique vertex $u$ adjacent to $w$ is different from all vertices of $T$.  Define  $T^\prime = S(T,\Lambda;v,w)$. Now it is easy to see that $Aut(T^\prime) \cong G_v$ and so $G_v \in \mathcal{T}$.
\end{proof}

Suppose $\Delta$ is an arbitrary bicyclic graph. Then the graph $\Delta$ has one of the following forms:
\begin{enumerate}
\item There are two cycles in $\Delta$ with at least one common edge.

\item There are two cycles in $\Delta$ without common edges and common vertices;

\item There are two cycles in $\Delta$ with a common vertex and without common edges;
\end{enumerate}
A bicyclic graph $H$ is said to be of type $i$  ($i=1, 2, 3$) if $H$  satisfies  the condition $i$.

Suppose $\Delta$ is a graph and $T_1$, $T_2$ are two subgraphs of $\Delta$ such that $T_1, T_2$ are trees and $v_1 \in V(T_1)$, $v_2 \in V(T_2)$ are vertices of a cycle in $\Delta$.  We say these trees satisfy the condition $(\star)$ if and only if ($T_1$, $T_2$) and ($Aut(T_1)_{v_1}$, $Aut(T_2)_{v_2}$) are pairs of isomorphic graphs.

\begin{lem} \label{lem2} Suppose $\Delta$  is a bicyclic graph depicted in Figure \ref{1} and all pairs of  elements in each set $\{ T_1, T_2, T_3, T_4\}$, $\{ \Theta_5, \Theta_6\}$ and $\{ \Upsilon_1, \Upsilon_2 \}$ satisfy the condition $(\star)$. We also assume that $D$ $=$ $(Aut(T_1))_a$ $\cong$ $(Aut(T_2))_b$ $\cong$ $(Aut(T_3))_c$ $\cong$ $(Aut(T_4))_d$, $K$ $=$ $(Aut(\Upsilon_1))_u$ $\cong$ $(Aut(\Upsilon_2))_v$ and $H=(Aut(\Theta_5))_{e_1}$ $\cong$ $(Aut(\Theta_6))_{e_2}$. Then,
$Aut(\Delta)$ $\cong$ $(D \times D \times D \times D \times H \times H \times K \times K) \rtimes_{\phi} (\Bbb{Z}_2 \times \Bbb{Z}_2)$.
\end{lem}

\begin{proof}
Define  $\phi:\Bbb{Z}_2 \times \Bbb{Z}_2 \to Aut(D \times D \times D \times D \times H \times H \times K \times K)$ by $\phi(0,0) =I$ and $\phi(0,1), \phi(1,1)$, $\phi(1,1)$ are defined as follows:
\begin{align*}
\phi(0,1)&=\psi_1:(\alpha,\beta,\gamma,\delta,\lambda,\lambda^{\prime},\sigma,\sigma^{\prime})\mapsto (\gamma,\delta,\alpha,\beta,\lambda,\lambda^{\prime},\sigma^{\prime},\sigma)\\
\phi(1,0)&=\psi_2:(\alpha,\beta,\gamma,\delta,\lambda,\lambda^{\prime},\sigma,\sigma^{\prime})\mapsto (\beta, \alpha, \delta,\gamma, \lambda^{\prime},\lambda,\sigma,\sigma^{\prime})\\
\phi(1,1)&=\psi_3:(\alpha,\beta,\gamma,\delta,\lambda,\lambda^{\prime},\sigma,\sigma^{\prime})\mapsto (\beta, \alpha, \delta,\gamma, \lambda^{\prime},\lambda,\sigma^{\prime},\sigma)
\end{align*}
where $\alpha \in Aut(T_1)$, $\beta \in Aut(T_2)$, $\gamma \in Aut(T_3)$, $\delta \in Aut(T_4)$, $\lambda \in Aut(\Upsilon_1)$, $\lambda^{\prime}\in Aut(\Upsilon_2)$, $\sigma \in Aut(\Theta_5)$ and $\sigma^{\prime} \in Aut(\Theta_6)$. Moreover, we assume that $V(T_i)$ $=$ $\{t_{1}^{i},t_2^i,\cdots, t_m^i\}$,
$V(\Theta_j)$ $=$ $\{s_{1}^{j},s_2^j,\cdots, s_p^j\}$ and $V(\Upsilon_k)$ $=$ $\{r_{1}^{k},r_2^k,\cdots, r_q^k\}$, where $1 \leq i \leq 4$, $j=5,6$ and $k=1,2$.

There are three paths connecting vertices $u$ and $v$. These paths have the vertex sets $V(P_1) = \{v,a,e_1,b,u\}$, $V(P_2) = \{ u,d,e_2,c,v\}$ and
$$V(P_3) = \left\{ \begin{array}{ll} \{ v, u_{11}, u_{12}, \ldots, u_{1t},u_{21}, u_{22}, \ldots, u_{2t}, u\} & 2 \nmid l(P_3)\\
\{ v, u_{11}, u_{12}, \ldots, u_{1t}, z, u_{21}, u_{22}, \ldots, u_{2t}, u\} & 2 \mid l(P_3)
\end{array}\right..$$

\begin{figure}
\centering
\includegraphics[width=12cm]{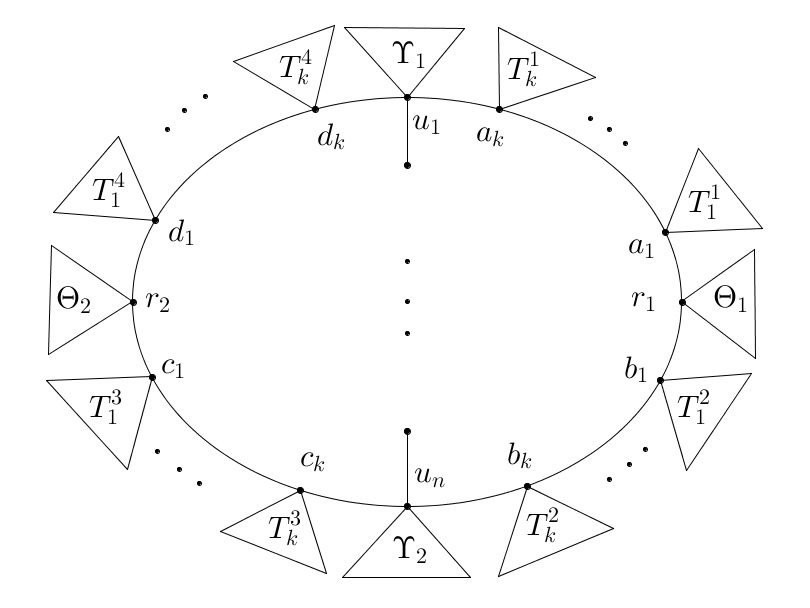}
\caption{A general bicyclic graph of the first type.} \label{2}
\end{figure}

Suppose  $O_1$ $=$ $V(\Delta) \setminus \left(V(T_1) \cup V(T_2) \cup V(T_3) \cup V(T_4) \cup V(\Theta_5) \cup V(\Theta_6)\right)$, $O_2$ $=$ $V(P_3) \setminus \{ u, v\}$,   $\sigma = (2 3)(1 4)(5 6)$ and $\tau = (1 2)(3 4)$. For $i = 1, 2, 3, 4$, $j = 5,6$ and $k = 1, 2$, we define two permutations $f_1$ and $f_2$ on $V(\Delta)$ and eight sets $V_i$, $U_j$ and $V_k$  as follows:

\begin{align*}
f_1 &= \left( \begin{array}{lll}  s_l^j & t_l^k & x\\ s_l^{\sigma(j)} & t_l^{\sigma(k)} & x \end{array}\right); \ \  x \in O_1, \\ f_2 &= \left( \begin{array}{llll} u_{lt}& r_l^j & t_l^k & s_l^j\\u_{\tau(l)t}& r_l^{\tau(j)} & t_l^{\tau(k)} & s_l^j \end{array}\right),\\
V_i&=\{f \in Aut(\Delta) \mid \forall x \in V({T_i}), f(x)=x \ \& \ f(t_1^i)=t_1^i\}; \ \ i = 1, 2, 3, 4, \\
U_j &=\{f \in Aut(\Delta) \mid \forall x \in V({\Theta_j}), f(x)=x \ \& \ f(s_1^j)=s_1^j\}; \ \ j = 5, 6, \\
G_k&=\{f \in Aut(\Delta) \mid \forall x \in V({\Upsilon_k}), f(x)=x \ \& \ f(r_1^k)=r_1^k\}; \ \ k = 1, 2.
\end{align*}

It is easy to see that $f_1, f_2$ are involutions in $Aut(\Delta)$. Define $L= \{ I,f_1\}$ and $M = \{ I,f_2\}$. Note that,

   \begin{align*}
   f_{1} o f_{2} (x)=
   \begin{cases}
   f_{1}(x) & \: x=s_l^j \in S_j; \: \: \: j=5,6 \\
   f_{2}(x) & \: x=r_l^t \in R_t; \: \: \: t=1,2 \\
   f_{1} f_{2} (t_l^1)=t_l^ {\sigma\tau(1)}=t_l^3  &\: x=t_l^1 \in T_1 \\
   f_{1} f_{2} (t_l^2)=t_l^ {\sigma\tau(2)}=t_l^4  &\: x=t_l^2 \in T_2 \\
   f_{1} f_{2} (t_l^3)=t_l^ {\sigma\tau(3)}=t_l^1  &\: x=t_l^3 \in T_3 \\
   f_{1} f_{2} (t_l^4)=t_l^ {\sigma\tau(4)}=t_l^2  &\: x=t_l^4 \in T_4
     \end{cases},
   \end{align*}
     \begin{align*}
  f_{2} o f_{1}(x)=
\begin{cases}
  f_{1}(x) & \: x=s_l^j \in S_j; \: \: \: j=5,6 \\
  f_{2}(x) & \: x=r_l^t \in R_t; \: \: \: t=1,2 \\
  f_{2} f_{1}  (t_l^1)=t_l^ {\tau\sigma(1)}=t_l^3  &\: x=t_l^1 \in T_1 \\
  f_{2} f_{1}  (t_l^2)=t_l^ {\tau\sigma(2)}=t_l^4  &\: x=t_l^2 \in T_2 \\
  f_{2} f_{1}  (t_l^3)=t_l^ {\tau\sigma(3)}=t_l^1  &\: x=t_l^3 \in T_3 \\
  f_{2} f_{1}  (t_l^4)=t_l^ {\tau\sigma(4)}=t_l^2  &\: x=t_l^4 \in T_4
  \end{cases},
  \end{align*}
 and so $f_1of_2 = f_2of_1$. This  proves that $ML = LM \cong L \times M$. We are now ready to prove that $V_1 V_2 V_3 V_4 U_5 U_6 G_1 G_2 \leq Aut(\Delta)$. If  $f \in V_i$ and  $g \in G_k$,  $1 \leq i \leq 4$ and $k = 1,2$,  then
$$fog(x)=
\begin{cases}
f(x) & x\in T_i \\
g(x) & x \in \Upsilon_j\\
x & otherwise
\end{cases}
= gof(x).$$

Thus  elements of $V_i$ and $G_k$  are commute to each other and so $V_iG_k$ is a subgroup of $Aut(\Delta)$. A similar argument shows that each element of $A$ commutes with each elements of $B$ such that  $A, B$ $\in$  $\Gamma_1 = \{ V_1, V_2, V_3, V_4, U_5, U_6, G_1, G_2 \}$. This proves that   $\Gamma_2 = V_1 V_2 V_3 V_4 U_5 U_6 G_1 G_2$ is a subgroup of $Aut(\Delta)$ and since each element of $\Gamma_1$ is a normal subgroup of $\Gamma_2$,
$V_1 V_2 V_3 V_4 U_5 U_6 G_1 G_2 \simeq V_1 \times V_2 \times V_3 \times V_4 \times U_5 \times U_6 \times G_1 \times G_2$. On the other hand, $V_i \simeq (Aut(T_i))_{t_1^i}$, $U_j  \simeq (Aut(\Theta_j))_{s_1^j}$ and $G_k   \simeq (Aut(\Upsilon_k))_{r_1^k}$.

We claim that if $\xi \in L$,  $\zeta \in M$, and  $\varrho \in Aut(\Delta)$, then $\zeta^{-1} \xi^{-1}\varrho(x) =\varrho \zeta^{-1} \xi^{-1}(x) $. To prove, we know that $\varrho(\Theta_j)\in \{\Theta_j, \Theta_{\sigma(j)}\}$,
$\varrho(\Upsilon_k)\in\{\Upsilon_k,\Upsilon_{\tau(k)}\}$ and
$\varrho(T_i)\in\{T_i, T_{\sigma(i)}, T_{\tau(i)}, T_{\sigma\tau(i)}\}$. If $\xi$ $=$ $\zeta = I$ then our claim is clear. We have three other cases as follows:
\begin{enumerate}
\item[a] $\xi=f_1 \in L$ and $\zeta =f_2 \in M$. Then   $\varrho(\Theta_j)=\Theta_{\sigma(j)}$,    $\varrho(\Upsilon_k)=\Upsilon_{\tau(k)}$  and  $\varrho(T_i)=T_{\sigma\tau(i)}$.  It is enough to show that $f_2^{-1}f_1^{-1}\varrho(x)=\varrho(f_2^{-1}f_1^{-1}(x))$. If $x=s_l^j\in\Theta_j$ and $\varrho(s_l^j)=s_{l^\prime}^{j^\prime}$, then
 \begin{align*}
 f_2^{-1}(f_1^{-1}\varrho(s_l^j))&= f_2^{-1}(s_{l^\prime}^{\sigma(j^\prime)})=s_{l^\prime}^{\sigma(j^\prime)}=s_{l^\prime}^j\\
 \varrho(f_2^{-1}f_1^{-1}(s_l^j))&=\varrho(s_l^{\sigma(j)})=s_{l^\prime}^j.
 \end{align*}

If $x=r_l^k\in\Upsilon_k$ and $\varrho(r_l^k)=r_{l^\prime}^{k^\prime}$, then
  \begin{align*}
 f_2^{-1}(f_1^{-1}\varrho(r_l^k))&= f_2^{-1}(r_{l^\prime}^{\tau(k^\prime)})=r_{l^\prime}^{\tau(k^\prime)}=r_{l^\prime}^k\\
 \varrho(f_2^{-1}f_1^{-1}(r_l^k))&=\varrho(r_l^{\tau(k)})=r_{l^\prime}^k.
 \end{align*}

If $x=t_l^i \in T_i$  and  $\varrho(t_l^i)= t_{l^\prime}^{i^\prime}$, then
 \begin{align*}
 f_2^{-1}(f_1^{-1}\varrho(t_l^i))&= f_2^{-1}( t_{l^\prime}^{\sigma(i^\prime)}= t_{l^\prime}^{\tau \sigma(i^\prime)}= t_{l^\prime}^i  \\
 \varrho(f_2^{-1}f_1^{-1}(t_l^i))&=\varrho(t_l^{\tau\sigma(i)})=t_{l^\prime}^{\sigma\tau\tau\sigma(i)}=t_{l^\prime}^i.
 \end{align*}
This completes the proof of this case.

 \item[b] Suppose  {$\xi=f_1 \in L$} and  {$\zeta = I \in M$}. Then  $\varrho(\Theta_j)=\Theta_{\sigma(j)}$, $\varrho(\Upsilon_k)=\Upsilon_{k}$ and $\varrho(T_i)=T_{\sigma(i)}$. It is enough to show that
 $f_1^{-1}\varrho(x)=\varrho(f_1^{-1}(x))$. If  {$x=s_l^j\in\Theta_j$} and {$\varrho(s_l^j)=s_{l^\prime}^{j^\prime}$}, then
\begin{align*}
(f_1^{-1}\varrho(s_l^j))&= (s_{l^\prime}^{\sigma(j^\prime)})=s_{l^\prime}^{\sigma(j^\prime)}=s_{l^\prime}^j\\
 \varrho(f_1^{-1}(s_l^j))&=\varrho(s_l^{\sigma(j)})=s_{l^\prime}^j.
 \end{align*}
 If {$x=r_l^k\in\Upsilon_k$} and {$\varrho(r_l^k)=r_{l^\prime}^{k^\prime}$}, then
 \begin{align*}
(f_1^{-1}\varrho(r_l^k))&=f_1(r_{l^\prime}^{\sigma(k)})=r_{l^\prime}^k\\
 \varrho(f_1^{-1}(r_l^k))&=\varrho(r_l^{\sigma(k)})=r_{l^\prime}^k.
 \end{align*}
 If {$x=t_l^i \in T_i$}  and  {$\varrho(t_l^i)= t_{l^\prime}^{i^\prime}$}, then
 \begin{align*}
(f_1^{-1}\varrho(t_l^i))&=f_1(  t_{l^\prime}^{\sigma(i^\prime)})= t_{l^\prime}^{ \sigma(i^\prime)}= t_{l^\prime}^i\\
 \varrho(f_1^{-1}(t_l^i))&=\varrho(t_l^{\sigma(i)})=t_{l^\prime}^{\sigma\sigma(i)}=t_{l^\prime}^i.
 \end{align*}
This completes the proof of this case.

 \item[c] Suppose that   {$\xi=I \in L$} and  {$\zeta =f_2 \in M$}. Then  {$\varrho(\Theta_j)=\Theta_{j}$},  {$\varrho(\Upsilon_k)=\Upsilon_{\tau(k)}$} and  {$\varrho(T_i)=T_{\tau(i)}$}. If  {$x=s_l^j\in\Theta_j$} and  {$\varrho(s_l^j)=s_{l^\prime}^{j^\prime}$}, then
 \begin{align*}
 f_2^{-1}(f_1^{-1}\varrho(s_l^j))&= s_{l^\prime}^{j}\\
 \varrho(f_2^{-1}f_1^{-1}(s_l^j))&=\varrho(s_l^{j})=s_{l^\prime}^j.
 \end{align*}
If  {$x=r_l^k\in\Upsilon_k$} and {$\varrho(r_l^k)=r_{l^\prime}^{k^\prime}$}, then
  \begin{align*}
 f_2^{-1}(f_1^{-1}\varrho(r_l^k))&= f_2^{-1}(r_{l^\prime}^{\tau(k^\prime)})=r_{l^\prime}^{\tau(k^\prime)}=r_{l^\prime}^k\\
 \varrho(f_2^{-1}f_1^{-1}(r_l^k))&=\varrho(r_l^{\tau(k)})=r_{l^\prime}^k.
 \end{align*}
 If  {$x=t_l^i \in T_i$} and  {$\varrho(t_l^i)= t_{l^\prime}^{i^\prime}$}, then
 \begin{align*}
 f_2^{-1}(f_1^{-1}\varrho(t_l^i))&= f_2^{-1}( t_{l^\prime}^{i^\prime}= t_{l^\prime}^{\tau (i^\prime)}= t_{l^\prime}^i  \\
 \varrho(f_2^{-1}f_1^{-1}(t_l^i))&=\varrho(t_l^{\tau(i)})=t_{l^\prime}^{\tau\tau(i)}=t_{l^\prime}^i.
 \end{align*}
 Therefore, the conditions of this case also lead to our desired result.
 \end{enumerate}

Now we show that  every $\varrho \in Aut(\Delta)$ can be written in the form      $\phi_1\phi_2\phi_3\phi_4h_5h_6g_1g_2\xi\zeta$ in such a way that
$(\phi_i, g_k, h_j, \xi, \zeta) \in V_i \times G_k \times U_j \times L \times M,$
where $1 \leq i \leq 4$, $k = 1,2$, $j = 5,6$, $\xi \in L$ and $\zeta \in M$. Define:

\begin{center}
$\phi_i(x)=\begin{cases}
\zeta^{-1} \xi^{-1}\varrho(x) =\varrho \zeta^{-1} \xi^{-1}(x)  &     x=t_l^i \in T_i/\{t_1^i\}\\
x & otherwise
\end{cases}   \in V_i$ \  \ $h_j(x) = \begin{cases}
\zeta^{-1} \xi^{-1}\varrho(x) =\varrho \zeta^{-1} \xi^{-1}(x)    & x=s_l^j \in \Theta_j/\{s_1^j\}\\
x & otherwise
\end{cases}   \in U_j$
$g_k(x)=\begin{cases}
\zeta^{-1} \xi^{-1}\varrho(x) =\varrho \zeta^{-1} \xi^{-1}(x)  & x=r_l^k \in \Upsilon_k/\{r_1^k\}\\
x  & otherwise
\end{cases} \in G_k$
\end{center}

We are ready to prove  that $ \phi_i$ is an automorphism. To prove $\phi_i$ is one to one, we assume that $x, x^\prime \in V(\Delta)$ with $x \ne x^\prime$ are arbitrary. We have to show that $\phi_i(x) \ne  \phi_i(x^\prime)$. To do this, the following two cases will be considered:
\begin{enumerate}
\item[(i)] $x, x^\prime \in T_i$. Since $\varrho, \xi, \zeta$ are permutations of $V(\Delta)$, $\phi_i(x)$ $=$ $\varrho\zeta^{-1}\xi^{-1}(x)$ $\ne$ $\varrho\zeta^{-1}\xi^{-1}(x^\prime)$ $=$ $\phi_i(x^\prime)$, as desired.

\item[(ii)] $x\in T_i, x^\prime \not\in \varrho(T_i)$. If $\varrho(T_i) = T_i$, then will have the case (i) and there is noting to prove. Suppose that $\varrho(T_i) \ne T_i$. This implies that $\varrho(T_i) \in \{ T_{\sigma(i)}, T_{\tau(i)}, T_{\sigma\tau(i)} \}$. If $x^\prime \in \varrho(T_i)$, $\xi\in L$ and $\zeta\in M$ then $\phi_i(x^\prime) = x^\prime \not\in T_i$ and $\phi_i(x) = \zeta^{-1}\xi^{-1}\varrho(x)  \in T_i$ and so $\phi_i(x) \ne \phi_i(x^\prime)$, as desired.
\end{enumerate}

Next we prove that $\phi_i$ is homomorphism. To do this, we assume that $u$ and $v$ are adjacent in $\Delta$. Then one of the following cases will be occurred:
\begin{enumerate}
\item If $u, v \in V(T_i)$, then $\phi_i(uv)=\zeta^{-1} \xi^{-1}\varrho(uv)=\zeta^{-1} \xi^{-1}\varrho (u) \zeta^{-1} \xi^{-1}\varrho  (v)=\phi_i(u)\phi_i(v) \in E(\Delta)$. Since $\varrho,$ $\zeta$ and $\xi$ are automorphism, they preserve adjacency in $\Delta$ and so $\phi_i$ has the same property.

\item If $v \not\in T_i$ and $u \in T_i$, then $ u = t_1^i$ and $\phi_i(uv)=uv=\phi_i(u)\phi_i(v)\in E(\Delta)$, as desired.

\item If $u, v \not\in T_i$, then $\phi_i(uv)=uv=\phi_i(u)\phi_i(v)\in E(\Delta)$.
\end{enumerate}
Next, we prove that
{$\phi_i^{-1}(x)=\begin{cases}
 \xi \zeta\varrho^{-1}(x)  &     x=t_l^i \in T_i/\{t_1^i\}\\
x & otherwise
\end{cases}$}
is also graph homomorphism. To do this, we assume that $u$ and $v$ are adjacent vertices in $\Delta$. Then, one of the following three cases can be occurred:
\begin{enumerate}
\item[(I)] $u,v \in V(T_i)$. Since $\xi$, $\zeta$ and $\varrho$ are automorphism, {$\phi_i^{-1}(uv)=\xi\zeta \varrho^{-1}(uv)= \xi\zeta\varrho^{-1} (u) \xi \zeta\varrho^{-1}  (v)=\phi_i^{-1}(u)\phi_i^{-1}(v) \in E(\Delta)$}, as desired.

\item[(II)]  $u\in T_i$ and $v \notin T_i $. In this case,  {$u=t_1^i$} and
{$\phi_i^{-1}(uv)=uv=\phi_i^{-1}(u)\phi_i^{-1}(v)\in E(\Delta)$}.

\item[(IV)] {$u,v \notin T_i$}. As similar argument as above shows that  {$\phi_i^{-1}(uv)=uv=\phi_i^{-1}(u)\phi_i^{-1}(v)\in E(\Delta)$}.
\end{enumerate}
To complete the proof, we note that
\begin{eqnarray*}
\phi_1 \phi_2 \phi_3\phi_4 h_5 h_6 g_1 g_2 \xi  \zeta  (t_k^i)
&=&\varrho \zeta^{-1} \xi^{-1}\xi \zeta (t_k^i)=\varrho(t_k^i)\\
\phi_1 \phi_2 \phi_3\phi_4 h_5 h_6 g_1 g_2 \xi  \zeta   (s_k^j) &=&\varrho \zeta^{-1} \xi^{-1}\xi \zeta (s_k^j) =\varrho(s_k^j)\\
\phi_1 \phi_2 \phi_3\phi_4 h_5 h_6 g_1 g_2 \xi  \zeta   (r_k^l)&=&\varrho \zeta^{-1} \xi^{-1}\xi \zeta  (r_k^l)=\varrho(r_k^l).
\end{eqnarray*}

This completes the proof.
\end{proof}

Define the functions $\xi_1, \xi_2 : \mathbb{N} \longrightarrow \mathbb{N}$ by
$$\xi_1(n) = \left \{ \begin{array}{ll}  \frac{n}{2}  & 2 \mid n \\ \frac{n-1}{2} & 2 \nmid n \end{array}\right. \ \ \text{and} \ \ \xi_2(n) = \left \{ \begin{array}{ll}  \frac{n}{2} + 1  & 2 \mid n \\ \frac{n+3}{2} & 2 \nmid n \end{array}\right..$$

\begin{cor} \label{cor2.3} Let $\Delta$ be an arbitrary bicyclic graph of the first type depicted in Figure \ref{2} and $T^i_j \cong T^r_j$, for each $i, j, r$ such that $1 \leq j \leq k$ and $1 \leq i,r \leq 4$. Then,
\begin{align*}
Aut(\Delta)&=(Aut(T_1^1))_{a_1} \times \cdots \times (Aut(T_k^1))_{a_k} \times (Aut(T_1^1))_{b_1} \times \cdots (Aut(T_k^2))_{b_k} \\
&\times
(Aut(\Upsilon_1))_{u_1} \times (Aut(\Upsilon_2))_{u_n} \times (Aut(T_1^3))_{c_1} \times \cdots \times (Aut(T_k^3))_{c_k}\\
&\times (Aut(T_1^4))_{d_1} \times \cdots \times (Aut(T_k^4))_{d_k} \times (Aut(\Theta_1))_{r_1} \times (Aut(\Theta_2))_{r_2} \rtimes_{\phi} \Bbb{Z}_2 \times \Bbb{Z}_2
\end{align*}
in which $\phi$ is a homomorphism from $\mathbb{Z}_2 \times \mathbb{Z}_2$ into $\mathcal{C}$ given by $\phi(0,0)=id$, $\phi(0,1)=\psi_1$,  $\phi(1,0)=\psi_2$ and $\phi(1,1)=\psi_3$. Here, $\mathcal{C}$,  $\psi_1$,  $\psi_2$ and $\psi_3$ are defined as follows:
\begin{align*}
\mathcal{C} &= Aut((Aut(T_1^1))_{a_1} \times \cdots \times (Aut(T_k^1))_{a_k} \times\cdots \times (Aut(\Upsilon_1))_{u_1} \times (Aut(\Upsilon_2))_{u_n} ),\\
\psi_1 &= (\alpha_1,\delta_1)...(\alpha_k,\delta_k)(\beta_1,\gamma_1)...(\beta_k,\gamma_k)(\mu_1,\mu_2)(u_1,u_n),(u_2,u_{n-1}) \cdots (u_{\xi_1(n)} u_{\xi_2(n)}), \\
\psi_2 &= (\alpha_1,\beta_1)\cdots (\alpha_k,\beta_k)(\gamma_1,\delta_1)\cdots(\gamma_k,\delta_k)(\epsilon_1,\epsilon_2)(u_1,u_n)(u_2,u_{n-1})\cdots(u_{\xi_1(n)} u_{\xi_2(n)}),\\
\psi_3 &= (\alpha_1,\gamma_1)\cdots(\alpha_k,\gamma_k)(\beta_1,\delta_1)\cdots(\delta_k,\beta_k)(\epsilon_1,\epsilon_2)(\mu_1,\mu_2)(u_1,u_n)\cdots (u_{\xi_1(n)} u_{\xi_2(n)}).
\end{align*}
\end{cor}

\begin{proof}
The induced subgraph of $\cup_{i=1}^kV(T_i^j)$ is denoted by $\Lambda^j$, $1 \leq j \leq 4$. By assumption $\Lambda^1 \cong \Lambda^2 \cong \Lambda^3 \cong \Lambda^4$ and all of them satisfy the condition $(\star)$. Apply Lemma \ref{lem2}, we have:

{\small\begin{eqnarray*}
Aut(\Delta)&=&\Big[(Aut(\Lambda^1)_{\{a_1,\cdots, a_k\} }\times Aut(\Lambda^2)_{\{b_1,\cdots, b_k\}} \times Aut(\Lambda^3)_{\{c_1,\cdots, c_k\}}\Big]\\ &\times& \Big[Aut(\Lambda^4)_{\{d_1,\cdots, d_k\}} \times Aut(\Theta_2)_{r_1} \times Aut(\Theta_3)_{r_2} \times Aut(\Upsilon_1)_{u_1} \times Aut(\Upsilon_2)_{u_n} \Big]\\ &\rtimes_{\phi}& \Bbb{Z}_2 \times  \Bbb{Z}_2,
\end{eqnarray*}}
proving the result.
\end{proof}

\begin{figure}
\centering
\includegraphics[width=7cm]{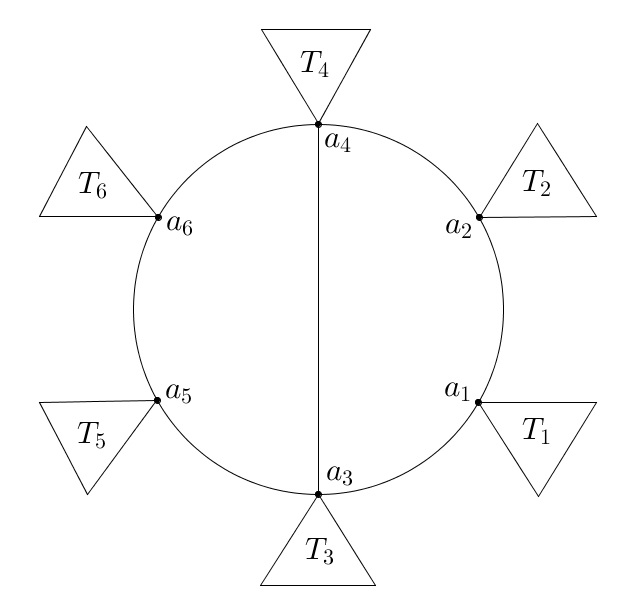}
\caption{The bicyclic graph of Lemma \ref{l2.3}.}\label{3}
\end{figure}

\begin{lem} \label{l2.3}
Suppose $T_1$, $T_2$, $\ldots$, $T_6$ are trees such that $T_1 \cong T_2$, $T_3 \cong T_4$, $T_5 \cong T_6$,  $G_1=Aut(T_1)_{a_1}\cong Aut(T_2)_{a_2}$, $G_2=Aut(T_3)_{a_3}\cong Aut(T_4)_{a_4}$  and $G_3 = Aut(T_5)_{a_5} \cong Aut(T_6)_{a_6}$, see Figure \ref{3}. Then,
$Aut(\Delta)=(G_1 \times G_2 \times G_3) \wr \Bbb{Z}_2$.
\end{lem}

\begin{proof}
Suppose $V(T_i)=\{t_1^i,\cdots, t_{k_i}^i\}$, $a_i = t_1^i$ and define $\sigma=(1 \ 2)(3 \ 4)(5 \ 6)$, $f_1 = (t_j^i   \ t_j^{\sigma(i)})$, $L=\{1,f_1\} $ and $ U_i=\{ \alpha \in Aut(\Delta) \mid \alpha (x) = x; \ x \notin T_i \ \& \  \alpha(t_1^i) = t_1^i \}$, $1 \leq i \leq 6$. Obviously,  $L$ and $ U_i$, $1 \leq i \leq 6$, are subgroups of $Aut(\Delta)$. It is easy to see that the mapping $\psi_i: Aut(T_i)_{a_i} \longrightarrow V_i$ given by $\psi_i(\alpha) = \alpha^\prime$ is an isomorphism in which $\alpha^\prime(x) = x,$ when $x = a_i$ or $x \not\in T_i$, and $\alpha^\prime(x) = \alpha(x)$, otherwise. Note that for each $f \in U_i$ and $h \in U_j$, $1 \leq i \ne j \leq 6$, $fh = hf$. This implies that $U = U_1U_2 \cdots U_6$ is a subgroup of $Aut(\Delta)$ and each subgroup $U_i$, $1 \leq i \leq 6$, is a normal subgroup of $U$. Since $U_i \cap U_1 \cdots U_{i-1}U_{i+1}\cdots U_6 = \{ id \}$, $U_1U_2 \cdots U_6 \cong U_1 \times \cdots \times U_6.$

To complete the proof, we show that $Aut(\Delta) = (U_1U_1 \cdots U_6)\cdot L$. To do this, we choose an arbitrary automorphism $\alpha$ in $Aut(\Delta)$.  Suppose  {$\alpha\in Aut(\Delta)$} and   {$\xi \in L$}  are arbitrary. We first show that  {$\alpha \xi^{-1}(x)= \xi^{-1} \alpha(x) $}. If {$\xi=I$} then obviously this equation is true. If {$\xi=f$}, then {$\alpha(T_i)=T_{\sigma(i)}$}. Assume that {$\alpha(t_l^i)=t_{l^\prime}^{\sigma(i)}$}. It is enough to show that ${\alpha f^{-1}(x) = f^{-1} \alpha(x)}$. To do this, we note that $\alpha f^{-1}(t_l^i)$ $=$ $\alpha(t_l^{\sigma(i)})$ $=$ $t_{l^{\prime}}^i$ and $f^{-1}\alpha(t_l^i)$ $=$ $f^{-1}(t_{l^\prime}^{\sigma(i)})$ $=$ $t_{l^\prime}^i$. Define:
$$\phi_i(x)=\begin{cases}
\alpha \xi^{-1}  (x)=\xi^{-1} \alpha(x)   &     x=t_k^i \in T_i/{t_1^i}\\
x  & otherwise
\end{cases}.$$

We claim that {$ \phi_i$} is an automorphism of $\Delta$. To prove {$ \phi_i$} is one to one, we assume that  {$x\neq x^{\prime}$}. We have two cases as follows:
\begin{enumerate}
\item[(I)]
$x,x^{\prime} \in T_i$. Since $\alpha$ and  $\xi$ are automorphism,   {$\alpha \xi^{-1}(x)\neq \alpha \xi^{-1}(x^\prime)$}, as desired.

\item[(II)] $x\in T_i$ and $x^{\prime}\notin T_i$. If {$\alpha(T_i)=T_i$}, then our we will have the case {$(I)$}. We  assume that {$\alpha(T_i)\neq T_i$}. Then $\alpha(T_i) = T_{\sigma(i)}$. If $x^{\prime}\in \alpha(T_i)$ and {$\xi \in L$}, then {$\phi_i(x^{\prime})=x^{\prime} \notin T_i$} and $\phi_i(x) = \xi^{-1}\alpha(x) \in T_i$ and so   $\phi_i(x^{\prime}) \neq \phi_i(x)$. If {$x^{\prime}\notin \alpha(T_i)$}, then $\phi_i(x^{\prime})=x^{\prime} \notin T_i$ and $\phi_i(x) = \xi^{-1}\alpha(x) \in T_i$. Again {$\phi_i(x^{\prime}) \neq \phi_i(x)$}, as desired.
\end{enumerate}
We are now ready to prove that $\phi_i$ is homomorphism. To see this, we assume that $u$ and $v$ are adjacent vertices of $\Delta$. Suppose {$u,v \in V(T_i)$}.  Then, $\phi_i(uv)= \xi^{-1}\alpha(uv)= \xi^{-1}\alpha (u) \xi^{-1}\alpha(v)$. Since both of $\xi$ and $\alpha$ are automorphism,  $\phi_i(uv)=\phi_i(u)\phi_i(v) \in E(\Delta)$, as desired. If $u\in T_i$ and $v \notin T_i $, then {$u=t_1^i$}   and $\phi_i(uv)=uv=\phi_i(u)\phi_i(v)\in E(\Delta)$, and if {$u,v \notin T_i$}, then  {$\phi_i(uv)=uv=\phi_i(u)\phi_i(v)\in E(\Delta)$}. This proves that  $\phi_i$ is homomorphism.
Next, we prove that {$\phi_i^{-1}(x)=\begin{cases}\xi \alpha^{-1}(x) &x=t_k^i \in T_i/{t_1^i}\\x  & otherwise
\end{cases}$} is also a homomorphism. Choose adjacent vertices $u, v \in V(\Delta)$.   Suppose {$u,v \in V(T_i)$}. Then, $\phi_i^{-1}(uv)= \xi\alpha^{-1}(uv)= \xi\alpha^{-1} (u) \xi\alpha^{-1}(v)$ and since both of $\xi$ and $\alpha$ are automorphism,  $\phi_i^{-1}(uv)=\phi_i^{-1}(u)\phi_i^{-1}(v) \in E(\Delta)$. We now assume that $u\in T_i$ and $v \notin T_i$. Then, {$u=t_1^i$} and we have $\phi_i^{-1}(uv)=uv=\phi_i^{-1}(u)\phi_i^{-1}(v)\in E(\Delta)$. If {$u,v \notin T_i$}, then $\phi_i^{-1}(uv)=uv=\phi_i^{-1}(u)\phi_i^{-1}(v)\in E(\Delta)$.

Hence $Aut(\Delta) = U_1U_2 \cdots U_6 \cdot L$.  Since $|L| = 2$ and   $U_1U_2 \cdots U_6$ $\cong$ $U_1 \times U_2 \times \cdots \times U_6$, $Aut(\Delta) = (U_1 \times U_2 \times \cdots \times U_6) \rtimes L$ $\cong$ $(G_1 \times G_2 \times G_3) \rtimes Z_2$, proving the lemma.
\end{proof}

\centering
\begin{figure}
\includegraphics[width=10cm]{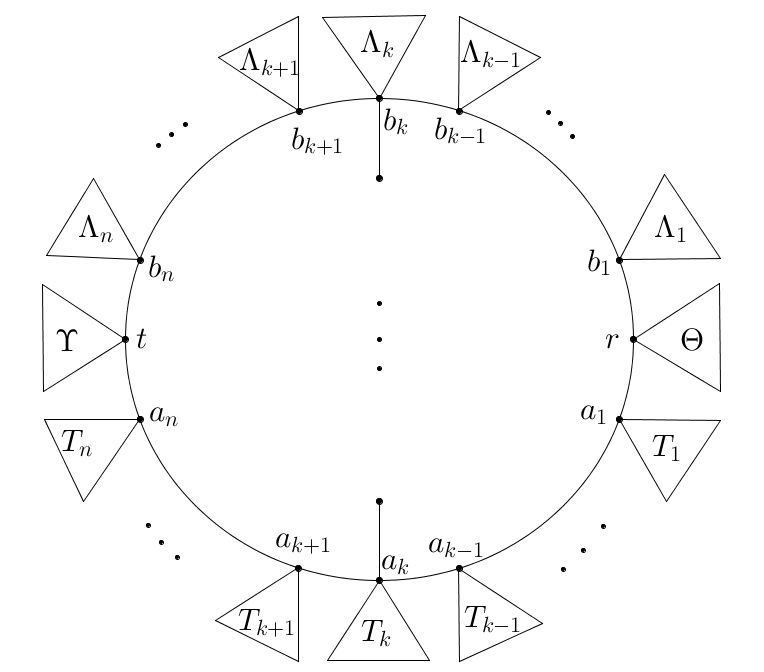}
\caption{ The graph $\overline{\mho}$ in Corollary \ref{cor2.4}.}
\label{4}
\end{figure}

\begin{cor} \label{cor2.4}
Suppose $G_i=(Aut(T_i))_{a_i}=(Aut(\Lambda_i))_{b_i}$, $1 \leq i \leq n$,  $H=(Aut(\Upsilon))_t$ and $K=(Aut(\Theta))_{r}$, where $t$ and $r$ are shown in the graph $\overline{\mho}$ depicted in Figure \ref{4}. Then,
$Aut(\Delta)=H \times K \times (G_1 \times G_2 \times \cdots \times G_n)\wr \Bbb{Z}_2$.
\end{cor}
\begin{center}
\begin{figure}
\includegraphics[width=6cm]{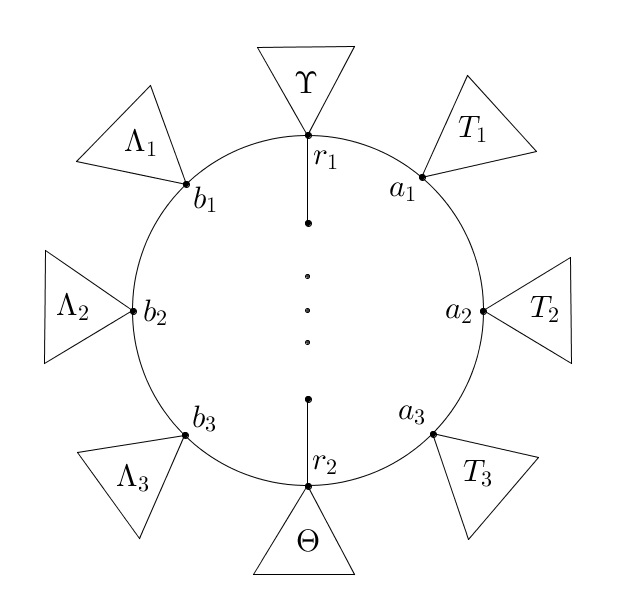}
\caption{ The figure for Lemma \ref{l2.5}.}\label{5}
\end{figure}
\end{center}

\begin{lem}\label{l2.5}
Suppose $G_i=(Aut(T_i))_{a_i}=(Aut(\Lambda_i))_{b_i}$, $H=(Aut(\Upsilon))_t$ and $K=(Aut(\Theta))_r$, see Figure \ref{5}. Then,
$Aut(\Delta)=H \times K \times (G_1 \times G_2 \times G_3) \wr \Bbb{Z}_2$.
\end{lem}

\begin{proof}
The proof is similar to the proof of Lemma \ref{l2.3} and so it is omitted.
\end{proof}

\begin{center}
\begin{figure}
\includegraphics[width=14cm]{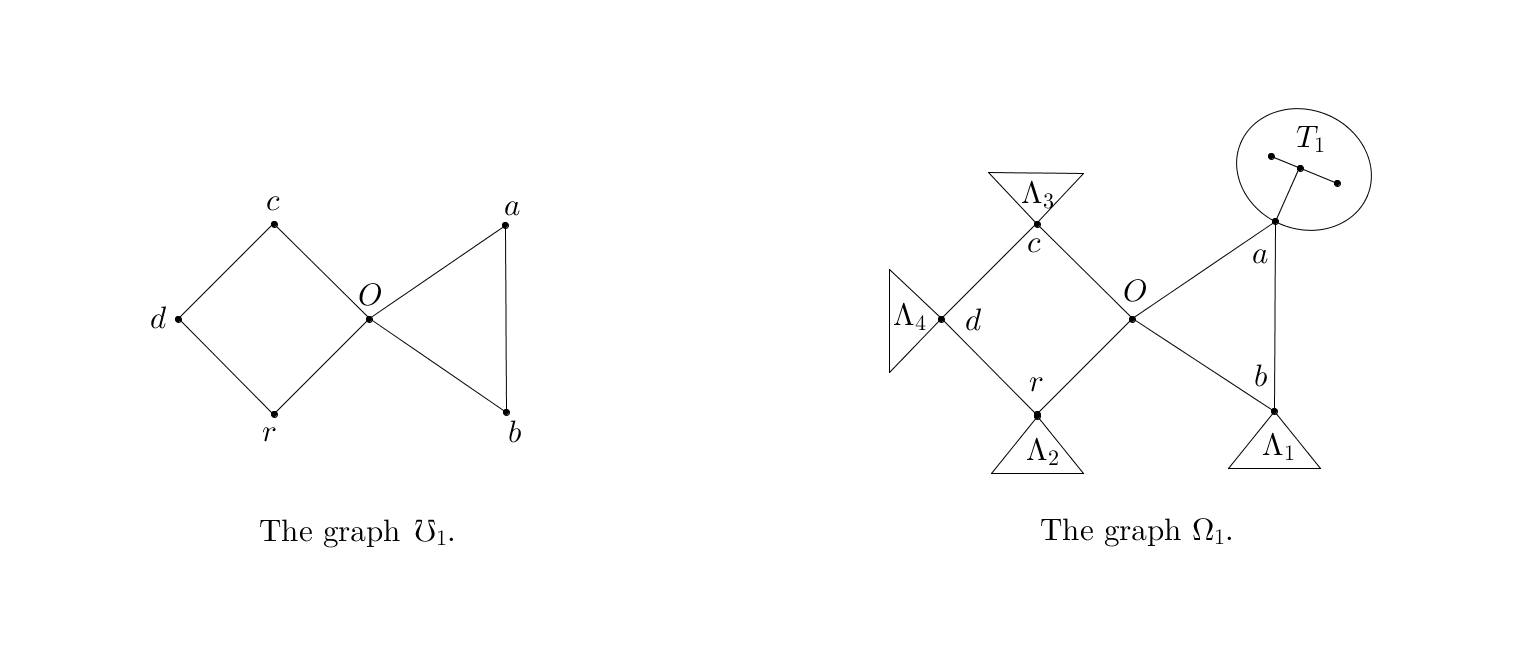}
\caption{Two graphs presented in the proof of Theorem 2.7.}
\label{6,7}
\end{figure}
\end{center}

Note that if a given tree $T$ has a central vertex $v$ then for each automorphism $\alpha \in Aut(T)$, $\alpha(v) = v$. For other type of trees, we will have the following definition.

\begin{defn}\label{d1} Suppose $G = Aut(T) \in \mathcal{T}$, $Fix(G) = \emptyset$ and $u, v$ are central vertices of $T$. It is well-known that $uv \in E(T)$. Add the vertex $u_T$ in the middle of $uv$, join vertices $u, v$ with $u_T$ and add another vertex $v_T$ together with the edge $u_Tv_T$ to construct a new tree $\overline{T}$.
\end{defn}

\begin{rem} By Definition \ref{d1}, $V(\overline{T}) =  V(T) \cup \{ u_T, v_T\}$ and $E(\overline{T}) =  (E(T) \setminus \{ uv\}) \cup \{ uu_T,u_Tv,u_Tv_T\}$. Also, it is easy to see that $Aut(T) \cong Aut(\overline{T})$.
\end{rem}

\begin{thm}
Every member of  $\mathcal{S}$ is isomorphic to the automorphism group of a bicyclic graph.
\end{thm}

\begin{proof}
Suppose $W$ is an arbitrary element of $\mathcal{S}$ $=$ $\mathcal{T} \cup \mathcal{B}_1 \cup \mathcal{B}_2$. We first assume that $W \in \mathcal{T}$ and write  $W=Aut(T_1)$, where $T_1$ is a tree. There are two different cases that $Fix(T_1) \neq \emptyset$  or  $Fix(T_1)= \emptyset$.

\begin{enumerate}
\item $Fix(T_1) \neq \emptyset$. Suppose $v \in Fix(T_1)$. In the graph $\mho_1$ in Figure \ref{6,7}, we choose the cycles $C_4$ and $C_5$ containing a common vertex $O$ together with four non-isomorphic asymmetric trees $\Lambda_1$, $\Lambda_2$, $\Lambda_3$ and $\Lambda_4$ containing vertices $v_1$, $v_2$, $v_3$ and $v_4$, respectively. We also assume that $\Lambda_i \not\cong T_1$, for each natural number $i$ in $\{ 1, 2, 3, 4\}$. Unify mutually the vertices $(v,a), (v_1,b), (v_2,r), (v_3,c)$ and $(v_4,d)$ to construct a bicyclic graph $\Omega_1$. Then $Aut(\Omega_1) \cong (Aut(T_1))_v \cong W_v = W$.

\item $Fix(T_1) = \emptyset$. In this case the center of $T_1$ is containing a unique edge $\ell$. Add a vertex into the edge $\ell$ and connect it to the vertex $a$, see Figure \ref{6,7}. In a similar way as in Case $(1)$, we choose asymmetric trees $\Lambda_1$, $\Lambda_2$, $\Lambda_3$ and $\Lambda_4$ and form a graph $\Omega_1$ as in depicted in the Figure \ref{6,7}. From this figure one can be easily seen that  $Aut(\Omega_1) = Aut(T_1) = W$.
\end{enumerate}

Next we assume that  $W \in \mathcal{B}_1$. Then there are two trees $T_2$ and $\Psi_1$  such that $W \cong C_1 \times [D_2 \wr (\mathbb{Z}_2 \times \mathbb{Z}_2)]$, where $C_1 = Aut(\Psi_1)$ and $D_2 = Aut(T_2)$. Define the graph $\Omega_2$ as follows:

\begin{enumerate}
\item[I] \textit{$Fix(T_2)\neq \emptyset$ and $Fix(\Psi_1)\neq \emptyset$}. Choose  {$w_{T_2} \in Fix(T_2)$} and     {$w_{\Psi_1} \in Fix(\Psi_1)$}. Suppose   {${T_2}^1$}, {${T_2}^2$}, {${T_2}^3$}, {${T_2}^4$},  {${T_2}^5$}, {${T_2}^6$},  {${T_2}^7$} and  {${T_2}^8$} are eight isomorphic copies of $T_2$ in which  the image of {$w_{T_2} \in V(T_2)$} under these isomorphisms are   {$w_{T_2^1}$}, {$w_{T_2^2}$},  {$w_{T_2^3}$}, {$w_{T_2^4}$}, {$w_{T_2^5}$},  {$w_{T_2^6}$},  {$w_{T_2^7}$} and {$w_{T_2^8}$}, respectively. Define
    $\Omega_2$ $=$ $S(\mho_2$, $T_2^1$, $T_2^2, T_2^3, T_2^4$, $T_2^5, T_2^6$, $T_2^7, T_2^8, \Psi_1$; $a_1, w_{T_2^1}; a_2,w_{T_2^2}; a_3$, $w_{T_2^3}; a_4, w_{T_4^4}$; $b_1$, $w_{T_2^5}; b_2$, $w_{T_2^6}; b_3$, $w_{T_2^7}; b_4, w_{T_4^8}; c, w_{\Psi_1}).$

\item[II] \textit{$Fix(T_2) = \emptyset$ and $Fix(\Psi_1) \neq \emptyset$}. Consider the vertex {$w_{\Psi_1} \in Fix(\Psi_1)$} and the tree  {$\overline{T_2}$} in Definition \ref{d1}. We also assume that {$\overline{T_2}^1$},  {$\overline{T_2}^2$}, {$\overline{T_2}^3$}, {$\overline{T_2}^4$},  {$\overline{T_2}^5$}, {$\overline{T_2}^6$}, {$\overline{T_2}^7$} and  {$\overline{T_2}^8$} are eight isomorphic copies of {$\overline{T_2}$} and {$v_{T_2^1}$},  {$v_{T_2^2}$},  {$v_{T_2^3}$},  {$v_{T_2^4}$}, {$v_{T_2^5}$},  {$v_{T_2^6}$},  {$v_{T_2^7}$} and  {$v_{T_2^8}$} are the image of {$v_{T_2}$} under isomorphisms between  {$\overline{T_2}$} and trees {$\overline{T_2}^1$},  {$\overline{T_2}^2$}, {$\overline{T_2}^3$}, {$\overline{T_2}^4$},  {$\overline{T_2}^5$}, {$\overline{T_2}^6$}, {$\overline{T_2}^7$} and  {$\overline{T_2}^8$}, respectively. Define $\Omega_2$ $=$ $S(\mho_2, \overline{T}_2^1$, $\overline{T}_2^2$, $\overline{T}_2^3$, $\overline{T}_2^4, \overline{T}_2^5$, $\overline{T}_2^6, \overline{T}_2^7$, $\overline{T}_2^8$, $\Psi_1; a_1, v_{T_2^1}; a_2, v_{T_2^2}; a_3, v_{T_2^3}; a_4, v_{T_4^4}; b_1,v_{T_2^5}$; $b_2, v_{T_2^6}$; $b_3,v_{T_2^7}$; $b_4,v_{T_4^8};c,w_{\Psi_1})$.

\item[III] \textit{$Fix(D_2)\neq \emptyset$ and $Fix(C_1)=\emptyset$}. Suppose {$w_{T_2} \in Fix(D_2)$},  {$\overline{\Psi}_1$} and {$v_{\Psi_1}$} are those defined in Definition \ref{d1}. Define
    $\Omega_2$ $=$ $S(\mho_2,$ $T_2^1, T_2^2, T_2^3,$ $T_2^4$, $T_2^5$, $T_2^6,$ $T_2^7$, $T_2^8$, $\overline{\Psi}_1$; $a_1$, $w_{T_2^1}; a_2, w_{T_2^2}$; $a_3, w_{T_2^3}; a_4, w_{T_4^4};$ $b_1, w_{T_2^5}; b_2, w_{T_2^6}; b_3, w_{T_2^7}; b_4, w_{T_4^8}; c, v_{\Psi_1})$.

\item[IV] \textit{$Fix(D_2)= \emptyset$ and $Fix(C_1)=\emptyset$}. Consider the tree  {$\overline{T_2}$} as in Definition \ref{d1} together with eight isomorphic copies of this tree named as  {$\overline{T_2}^1$}, {$\overline{T_2}^2$}, {$\overline{T_2}^3$}, {$\overline{T_2}^4$},  {$\overline{T_2}^5$}, {$\overline{T_2}^6$}, {$\overline{T_2}^7$} and {$\overline{T_2}^8$}. Furthermore, we choose the vertices  {$v_{T_2^1}$}, {$v_{T_2^2}$}, {$v_{T_2^3}$},  {$v_{T_2^4}$},  {$v_{T_2^5}$}, {$v_{T_2^6}$},  {$v_{T_2^7}$} and {$v_{T_2^8}$} to be the image of  {$v_{T_2}$} under appropriate isomorphism between  {$\overline{T_2}$} and  {$\overline{T_2}^1$}, {$\overline{T_2}^2$}, {$\overline{T_2}^3$}, {$\overline{T_2}^4$},  {$\overline{T_2}^5$}, {$\overline{T_2}^6$}, {$\overline{T_2}^7$} and {$\overline{T_2}^8$}, respectively. We also assume that the tree {$\overline{\Psi}_1$} and the vertex {$v_{\Psi_1}$} are according to Definition \ref{d1}. Define
    $\Omega_2$ $=$ $S(\mho_2, \overline{T}_2^1, \overline{T}_2^2, \overline{T}_2^3, \overline{T}_2^4, \overline{T}_2^5, \overline{T}_2^6, \overline{T}_2^7, \overline{T}_2^8,$ $\overline{\Psi}_1; a_1, v_{T_2^1}; a_2, v_{T_2^2}; a_3, v_{T_2^3}; a_4, v_{T_4^4};$ $b_1, v_{T_2^5}; b_2, v_{T_2^6}; b_3, v_{T_2^7}; b_4, v_{T_4^8}; c, v_{\Psi_1}).$
    \end{enumerate}
In all cases,   {$Aut(\Omega_2)=C_1\times( D_2 \wr  (\Bbb{Z}_2 \times \Bbb{Z}_2))$}, as desired.

\hspace{5mm} Finally, we assume that $W\in \mathcal{B}_2$. Set $W=C_2\times(D_3 \times D_3 \times D_3 \times D_3 \times H \times H \times K \times K \rtimes (\Bbb{Z}_2 \times \Bbb{Z}_2)) $. Hence, there are trees  $T_3$, $\Lambda$, $\Upsilon$ and  $\Psi_2$ such that {$D_3 = Aut(T_3)$}, {$H=Aut(\Lambda)$},  {$K=Aut(\Upsilon)$} and $C_2=Aut(\Psi_2)$. Set $FE = \{ Fix(D_3), Fix(H), Fix(K), Fix(C_2)\}$. There are sixteen cases that for which one, two, three or all elements of $FE$ are empty.

\hspace{5mm} If $Fix(D_3) = \emptyset$, then we apply Definition \ref{d1} to construct  four copies  {$\overline{T_3}^1$}, {$\overline{T_3}^2$}, {$\overline{T_3}^3$} and {$\overline{T_3}^4$} of {$\overline{T}_3$} together with four vertices {$v_{T_3^1} \in V(\overline{T_3}^1)$}, {$v_{T_3^2} \in V(\overline{T_3}^2)$}, {$v_{T_3^3} \in V(\overline{T_3}^3)$} and  {$v_{T_3^4} \in V(\overline{T_3}^4)$} in which these four vertices are images of the vertex $v_{T_3}$ under some appropriate isomorphisms from $\overline{T}_3$ onto   {$\overline{T_3}^1$}, {$\overline{T_3}^2$}, {$\overline{T_3}^3$} and {$\overline{T_3}^4$}, respectively. If $Fix(D_3) \neq \emptyset$, then we will consider  four copies $T_3^1, T_3^2, T_3^3$ and $T_3^4$ of $T_3$ together with four vertices $w_{T_3^1} \in V({T_3}^1)$, $w_{T_3^2} \in V({T_3^2})$, $w_{T_3^3} \in V({T_3}^3)$ and $w_{T_3^4} \in V({T_3^4})$ in which these four vertices are images of the vertex  $w_{T_3} \in Fix(D_3)$ under some appropriate isomorphisms from ${T}_3$ onto   {${T_3}^1$}, {${T_3}^2$}, {${T_3}^3$} and {${T_3}^4$}, respectively.

\hspace{5mm} If $Fix(H) = \emptyset$, then we  construct  two copies  {$\overline{\Theta}^1$} and {$\overline{\Theta}^2$} of {$\overline{\Theta}$} together with two vertices $v_{\Theta^1} \in V(\overline{\Theta}^1)$, $v_{\Theta^2} \in V(\overline{\Theta}^2)$ in which these two vertices are images of the vertex $v_{\Theta} \in V(\overline{\Theta})$ under some appropriate isomorphisms from $\overline{\Theta}$ onto $\overline{\Theta}^1$ and $\overline{\Theta}^2$, respectively. If $Fix(H) \neq \emptyset$, then we consider two copies  {${\Theta}^1$} and {${\Theta}^2$} of {${\Theta}$} together with two vertices $w_{\Theta^1} \in V(\Theta^1)$, $w_{\Theta^2} \in V(\Theta^2)$ in which these vertices are images of the vertex $w_{\Theta} \in  Fix(H)$ under some appropriate isomorphisms from $\Theta$ onto $\Theta^1$ and $\Theta^2$, respectively.

\hspace{5mm} If $Fix(K) = \emptyset$, then we  construct  two copies  {$\overline{\Upsilon}^1$} and {$\overline{\Upsilon}^2$} of {$\overline{\Upsilon}$} together with two vertices $v_{\Upsilon^1} \in V(\overline{\Upsilon}^1)$, $v_{\Upsilon^2} \in V(\overline{\Upsilon}^2)$ in which these two vertices are images of the vertex $v_{\Upsilon} \in V(\overline{\Upsilon})$ under some appropriate isomorphisms from $\overline{\Upsilon}$ onto $\overline{\Upsilon}^1$ and $\overline{\Upsilon}^2$, respectively. If $Fix(K) \neq \emptyset$, then we consider two copies  {${\Upsilon}^1$} and {${\Upsilon}^2$} of {${\Upsilon}$} together with two vertices $w_{\Upsilon^1} \in V(\Upsilon^1)$, $w_{\Upsilon^2} \in V(\Upsilon^2)$ in which these vertices are images of the vertex $w_{\Upsilon} \in  Fix(K)$ under some appropriate isomorphisms from $\Upsilon$ onto $\Upsilon^1$ and $\Upsilon^2$, respectively.

\hspace{5mm} If $Fix(C_2) = \emptyset$ then we consider the vertex $v_{\Psi_2}$  by Definition \ref{d1}, and if $Fix(C_2)\neq \emptyset$ then we choose  $w_{\Psi_2} \in Fix(C_2)$. Next we define Set $\Omega_3$ $=$ $S(\mho_2,\widehat{T_3^1},$ $\widehat{T_3^2}, \widehat{T_3^2}, \widehat{T_3^4}$, $\widehat{\Theta^1}, \widehat{\Theta^2}$, $\widehat{\Upsilon^1}, \widehat{\Upsilon^2}$, $\widehat{\Psi_2};a_1$, $e_{T_3^1}$; $a_2,e_{T_3^2}$; $a_3,e_{T_3^3}$; $a_4,e_{T_3^4}$; $b_1,e_{\Theta^1}$; $b_2,e_{\Theta^2}$; $b_3,e_{\Upsilon^1}$; $b_4,e_{\Upsilon^2}$; $c,e_{\Psi_2})$. Here, for each tree $L$,
\begin{center}
$\widehat{L} = \left\{\begin{array}{ll} \overline{L} & Fix(L) = \emptyset\\ {L} & Fix(L) \neq \emptyset \end{array}\right.$ and $e_{L} = \left\{\begin{array}{ll} v_{L} & Fix(L) = \emptyset\\ w_{L} & Fix(L) \neq \emptyset \end{array}\right.$.
\end{center}
By our construction, $Aut(\Omega_3)=C_2\times( D_3 \times D_3 \times D_3 \times D_3 \times H \times H \times K \times K \rtimes (\Bbb{Z}_2 \times \Bbb{Z}_2))$ which completes the proof.
\end{proof}

\begin{center}
\begin{figure}
\includegraphics[width=8cm]{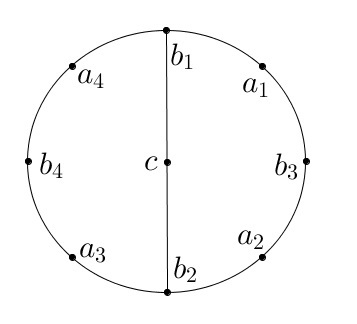}
\caption{The graph $\mho_2$.}
\label{8,9,10}
\end{figure}
\end{center}

\begin{thm}
 The automorphism group of every bicyclic graphs is a member of $\mathcal{S}$.
\end{thm}

\begin{proof}
There are three different types of bicyclic graphs as follows:
\begin{enumerate}
\item[(I)] \textit{There are two cycles in the graph with some common edges.} The result follows from Lemmas \ref{lem2}, \ref{l2.3}, \ref{l2.5} and Corollaries \ref{cor2.3}, \ref{cor2.4}.

\item[(II)] \textit{There are two cycles with a common vertex.} Suppose these two cycles have lengths $m$ and $n$, respectively. There are two different cases that the length of two cycles are equal or they have different lengths. We first assume that two cycles have the same length, i.e. $m=n$. From Figure \ref{11}, one can see that $m = \gamma + \gamma^\prime + 2$ and $n = \delta + \delta^\prime + 2$. There are five separate cases for the case that $m=n$  as follows:
\begin{figure}
\centering
\includegraphics[width=12cm]{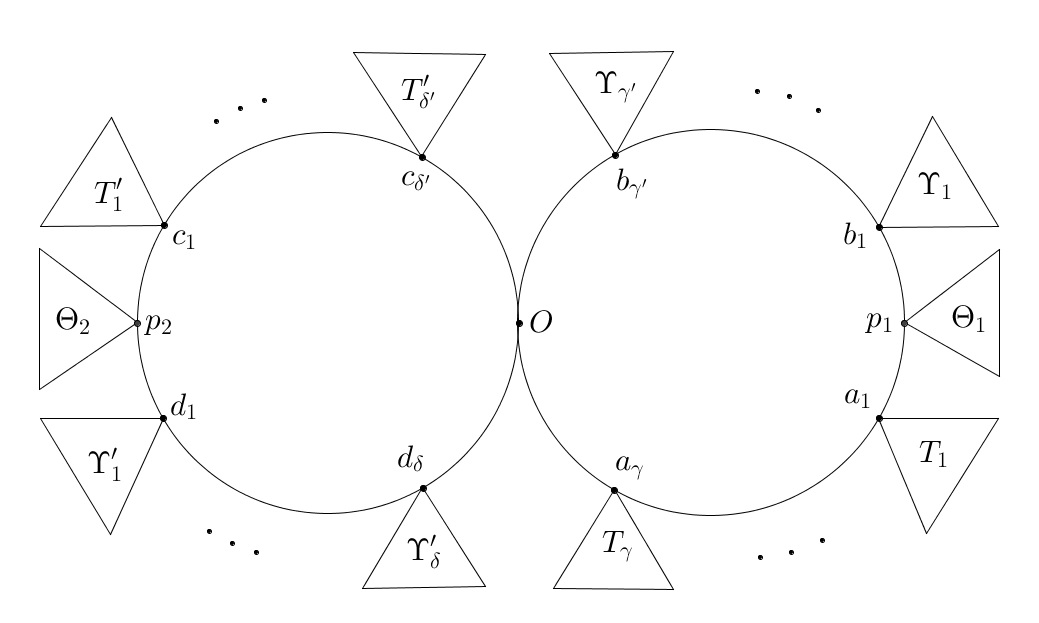}
\caption{The general case of a bicyclic graph when the cycles have a common vertex.}
\label{11}
\end{figure}
\begin{enumerate}
\item[(M1)] \textit{We have two cycles without trees attached to the vertices}. Suppose $\Delta$ is a bicyclic graph constructed from two cycles with a common vertex such that all vertices other than the common vertex have degree $2$, see Figure \ref{11} for details. Therefore, $Aut(\Delta)$ $=$ $\Bbb{Z}_2 \wr \Bbb{Z}_2 \in Aut(TREE)$ $\subseteq$ $\mathcal{S}$.

\item[(M2)] \textit{$(Aut(T_i))_{a_i}$ $\cong$ $(Aut(T^{\prime}_{i^\prime}))_{c_{i^\prime}}$ $\cong$ $(Aut(\Upsilon_j))_{b_j}$ $\cong$ $(Aut(\Upsilon^{\prime}_{j^\prime}))_{d_{j^\prime}}$ $\cong$ $(Aut(\Theta_1)_{p_1}$ $\cong$ $(Aut(\Theta_2)_{p_2}$ $\cong$ $G$, where $1 \leq i \leq \gamma$, $1 \leq i^\prime \leq \gamma^\prime$, $1 \leq j \leq \delta^\prime$ and $1 \leq j^\prime \leq \delta$}. We consider the bicyclic graph $\Delta$ in such a way that there are isomorphic rooted trees  $(T_1, a_1)$, $(T_{\gamma},a_{\gamma})$,  $(\Upsilon_1, b_1)$, $(\Upsilon_{\gamma^\prime}, b_{\gamma^\prime})$, $(T^{\prime}_1, c_1)$,  $(T^{\prime}_{\delta^\prime}, c_{\delta^\prime})$,  $(\Upsilon^{\prime}_1, d_1)$, $(\Upsilon^{\prime}_{\delta}, d_{\delta})$, $(\Theta_1, p_1)$, $(\Theta_2, p_2)$ attached to non-common vertices of two cycles which satisfy the condition $(\star)$, see Figure \ref{11}.  Then $Aut(\Delta)$ $\cong$ $G \wr (\Bbb{Z}_2 \wr \Bbb{Z}_2)$ $\in$ $\mathcal{T}$ $\subseteq$ $\mathcal{S}$.

\item[(M3)] Consider the bicyclic graph $\Delta$ with this property that $\gamma=\gamma^\prime=\delta=\delta^\prime$, $Aut(T^\prime_i)_{c_i}\cong Aut(\Upsilon_i)_{b_i} \cong Aut(\Upsilon_i^\prime)_{d_i} \cong Aut(T_i)_{a_i}$ and $Aut(\Theta_1)_{p_1} \cong Aut(\Theta_2)_{p_2}$. Define $G_i=(Aut(T_i))_{a_i}$ and $H_i=(Aut(\Theta_i))_{p_i}$. Then,
$Aut(\Delta)=(H_1 \times (G_1 \times \cdots G_{\gamma})\wr \Bbb{Z}_2) \wr\Bbb{Z}_2 \in \mathcal{T} \in \mathcal{S}$.

\item[(M4)] Consider the graph $\Delta$ in such a way that $\delta^\prime= \gamma^\prime$ and $\delta=\gamma$. Moreover, we assume that $\Upsilon_i \cong T^\prime_i$, $1\leq i  \leq \delta^\prime$, and they satisfy the condition $(\ast)$, $T_j \cong \Upsilon^\prime_j$, $1\leq j \leq \delta$, and again these graphs satisfy the condition $(\ast)$. By Figure \ref{11}, $  \Upsilon_i \cong T^{\prime}_i$, $T_i\cong \Upsilon^{\prime}_i$ and $\Theta_1\cong \Theta_2$.  Set $K_i=(Aut(\Upsilon_i))_{b_i}$. Then,
$Aut(\Delta)$ $\cong$ $(H_1\times G_1\times \cdots \times G_{\gamma}$ $\times$ $K_1 \times \cdots \times K_{\gamma^\prime})\wr \Bbb{Z}_2 \in  \mathcal{S}.$

\item[(M5)] In this case,  the general case of $M1-M4$ is considered into account in which we don't have isomorphisms between trees. Set $K_i=(Aut(\Upsilon_i))_{b_i}$, $G_i=(Aut(T_i))_{a_i}$, $G^{\prime}_i=(Aut(T^{\prime}_i))_{c_i}$ and $K^{\prime}_i=(Aut(\Upsilon^{\prime}_i)_{d_i}$. Then,
 $Aut(\Delta)$ $\cong$ $H_1 \times H_2 \times K_1 \times \cdots K_{\gamma^\prime}$ $\times$ $K^{\prime}_1 \times \cdots \times K^{\prime}_{\gamma}$ $\times$ $G_1 \times \cdots \times G_{\gamma}$ $\times$ $G^{\prime}_1  \times  \cdots \times G^{\prime}_{\delta^\prime}.$
\end{enumerate}

 \vskip 3mm

\noindent If two cycles have different lengths then we will have three cases as follows:

\begin{enumerate}
\item[(M6)] If there is no   tree $T$ such that $T$ is attached to a vertex of $\Delta$, then
 $Aut(\Delta)=\Bbb{Z}_2 \times \Bbb{Z}_2   \in \mathcal{S}$.

\item[(M7)] Suppose the graph $\Delta$ has this property that $\gamma=\gamma^\prime$ and $\delta= \delta^\prime$. Furthermore, we assume that for each $i$, $1\leq i \leq \gamma$, $\Upsilon_i \cong T_i$ satisfy the condition $(\ast)$ and for each $j$, $1\leq j \leq \delta$, $\Upsilon^\prime_j \cong T^\prime_j$ satisfy again $(\ast)$. Therefore, $Aut(\Delta)$ $\cong$ $(G_1\times\cdots \times G_{\gamma} \times  G^{\prime}_1\cdots \times G^{\prime}_{\delta})$ $\wr$ $\Bbb{Z}_2 \times H_1 \times H_2 $  $\in$ $\mathcal{S}$.

\item[(M8)] Suppose that $\gamma=\gamma^\prime$ and for each $i$, $1\leq i \leq \gamma $, $\Upsilon_i$,
$T_i$ are isomorphic and satisfy the condition $(\ast)$. Then it can be proved that $Aut(\Delta)$ $\cong$ $(G_1\times\cdots \times G_{\gamma})\wr \Bbb{Z}_2 \times H_1 \times H_2 \times G^{\prime}_1\cdots \times G^{\prime}_{\delta} \times K^{\prime}_1\cdots \times K^{\prime}_{\delta}$  $\in$ $\mathcal{S}$.
\end{enumerate}

\begin{figure}
\centering
\includegraphics[width=14cm]{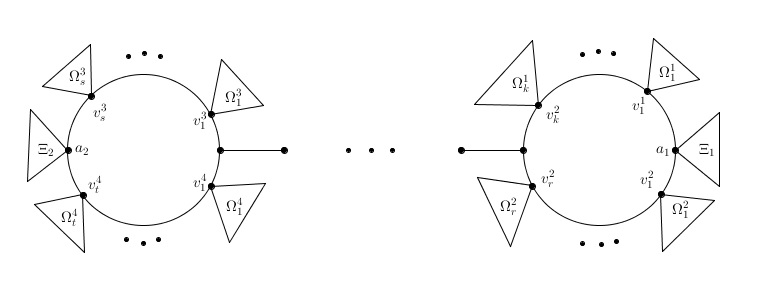}
\caption{A figure for the proof of Theorem 2.7 (Case III).}\label{12}
\end{figure}

\item[(III)] \textit{Two cycles of the graph is connected to each other by a path.} In this case, there are three cases for the bicyclic graph $\Delta$ and its general form is depicted in Figure \ref{12}. Suppose $F^j_i=Aut(\Omega^j_i)_{v^j_i}$ and $E_l=Aut(\Xi_l)_{a_l}$.
\begin{enumerate}
\item[(N1)] \textit{In the graph $\Delta$, $k = r = s = t$ and for each $i$, $\Omega_i^1$, $\Omega_i^2$, $\Omega_i^3$ and $\Omega_i^4$ are isomorphic and satisfy the condition $(\star)$, Figure \ref{12}. Moreover, $\Xi_1$ and $\Xi_2$ are isomorphic and satisfy again the condition $(\star)$}. By Figure \ref{12},  one can see that $Aut(\Delta)$ $\cong$ $(E_1 \times (F^1_1 \times \cdots \times F^1_k)\wr \Bbb{Z}_2 ) \wr  \Bbb{Z}_2 \in \mathcal{S}$, as desired.
\item[(N2)] \textit{In the graph $\Delta$, $k=r$ and $s=t$ and for each $i$, both $\Omega_i^1$, $\Omega_i^2$ and $\Omega_i^3$, $\Omega_i^4$ are mutually isomorphic and satisfy condition $(\star)$.} In this case, by Figure \ref{12} one can be easily seen that $Aut(B)$ $\cong$ $E_1 \times E_2 \times (F^1_1\times \cdots F^1_k)\wr \Bbb{Z}_2  \times  (F^3_1\times \cdots F^3_s)\wr \Bbb{Z}_2 \in \mathcal{S},$ that is our claim.
\item[(N3)] \textit{In the graph $\Delta$, $k = r = s = t$ and for each $i$, all pairs $\Omega_i^1$, $\Omega_i^3$; $\Omega_i^2$, $\Omega_i^4$ and $\Xi_1$, $\Xi_2$ are mutually isomorphic and all of them satisfy the condition $(\star)$.} Again we use the Figure \ref{12} to prove that $Aut(\Delta)$ $\cong$ $(E_1 \times F^1_1 \times \cdots \times F^1_k \times F^2_1 \times \cdots \times  F^2_r) \wr \Bbb{Z}_2 \in \mathcal{S}.$
\end{enumerate}
\end{enumerate}
Hence the result.
\end{proof}

\vskip 3mm

\noindent{\bf Acknowledgement.} The research of the authors are partially supported by the University of Kashan
under grant no 364988/111.

\end{document}